\documentclass{amsart}

\usepackage{amsmath,amssymb}
\usepackage{amsthm}
\usepackage{amsrefs}
\usepackage{indentfirst}
\usepackage{mathrsfs}
\usepackage{graphicx}

\usepackage[dvipsnames]{xcolor}
\usepackage[hidelinks]{hyperref}
\hypersetup{colorlinks = true,
linkcolor = MidnightBlue, anchorcolor =red,
citecolor = ForestGreen,
filecolor = red,urlcolor = red,
            pdfauthor={Zexi Fan, Bowen Li, Jianfeng Lu}}
\usepackage[capitalise]{cleveref}

\usepackage[margin=1.3in]{geometry}

\theoremstyle{plain}
\newtheorem{theorem}{Theorem}
\newtheorem{lemma}{Lemma}

\newtheorem{coro}[theorem]{Corollary}

\theoremstyle{definition}

\DeclareMathOperator{\tr}{tr}

\newcommand{\thmref}[1]{Theorem~\ref{#1}}
\newcommand{\corref}[1]{Corollary~\ref{#1}}

\newcommand{\ud}{\,\mathrm{d}}

\newcommand{\RR}{\mathbb{R}}

\IfFileExists{mathabx.sty}%
  {\DeclareFontFamily{U}{mathx}{\hyphenchar\font45}%
   \DeclareFontShape{U}{mathx}{m}{n}{<->mathx10}{}%
   \DeclareSymbolFont{mathx}{U}{mathx}{m}{n}%
   \DeclareFontSubstitution{U}{mathx}{m}{n}%
   \DeclareMathAccent{\widebar}{0}{mathx}{"73}%
}{%
  \PackageWarning{mathabx}{%
    Package mathabx not available, therefore\MessageBreak substituting
    widebar with overline\MessageBreak }%
  \newcommand{\widebar}[1]{\overline{#1}}%
}

\newcommand{\mc}[1]{\mathcal{#1}}

\newcommand{\mf}[1]{\mathsf{#1}}

\newcommand{\norm}[1]{\lVert#1\rVert}

\newcommand{\average}[1]{\langle#1\rangle}

\newcommand{\ip}[2]{\langle #1,#2\rangle}

\title[Sharp hypocoercive convergence estimates for underdamped Langevin dynamics]{Sharp hypocoercive convergence estimates for underdamped Langevin dynamics via the modified $L^2$ method}
\thanks{This work was partially supported by the National Key R\&D Program of China grant number 2024YFA1016000 (B.L.) and by the National Science Foundation via grant DMS-2309378 (J.L.). We thank anonymous reviewers for helpful suggestions.}

\author{Zexi Fan}
\address{School of Mathematical Sciences, Peking University}
\email{2200010816@stu.pku.edu.cn}

\author{Bowen Li}
\address{Department of Mathematics, City University of Hong Kong}
\email{boweli4@cityu.edu.hk}

\author{Jianfeng Lu}
\address{Department of Mathematics, Department of Physics, and Department of Chemistry, Duke University}
\email{jianfeng@math.duke.edu}

\date{\today}


\begin{document}

\begin{abstract}
In this note, we consider the underdamped Langevin dynamics with invariant measure $\mu(\mathrm{d}x\,\mathrm{d}v) \propto e^{-U(x)-|v|^2/2}\,\mathrm{d}x\,\mathrm{d}v$.
Assume that the position marginal $\mu_x(\mathrm{d}x)\propto e^{-U(x)}\,\mathrm{d}x$ satisfies a Poincar\'{e} inequality with constant $m>0$, and that $\nabla^2 U\ge -K\,\mathrm{Id}$ for some $K\ge 0$. We revisit the modified $L^2$ method of Dolbeault--Mouhot--Schmeiser, employing a shifted corrector
\begin{equation*}
    \mc A_\alpha=(\alpha-\mc L_{\mathrm o})^{-1}(\mc L_a\Pi_v)^*,
  \qquad \alpha\ge0,
\end{equation*}
where $\mc{L}_{\mathrm{o}}=\Delta_x-\nabla U\cdot\nabla_x$ is the overdamped generator, $\mc{L}_a$ is the generator of the Hamiltonian flow, and $\Pi_v$ denotes averaging over the velocity variable. We establish an explicit hypocoercive $L^2$-convergence rate $\Lambda_{\alpha,\gamma}$ for every shift $\alpha\ge0$ and friction coefficient $\gamma>0$, and show that, for each fixed $\gamma$, the rate is maximized at $\alpha=0$. Optimizing further over $\gamma$ gives
\begin{equation*}
    \Lambda_{0,\gamma_*}  = \frac{\sqrt m}
  {2\left(2+\sqrt{2+\frac{2K}{m}}
  +\sqrt{6+\frac{2K}{m}}\right)}\,, \qquad \text{at} \quad   \gamma_*=\sqrt{6m+2K}.
\end{equation*}
For convex $U$, this recovers the optimal $\mathcal{O}(\sqrt m)$ rate.
\end{abstract}

\maketitle

\section{Introduction}

We study the underdamped Langevin dynamics for $(X_t, V_t) \in \mathbb{R}^d \times \mathbb{R}^d$:
\begin{subequations}\label{eq:langevin}
\begin{align}
  \mathrm{d}X_t &= V_t\,\mathrm{d}t, \\
  \mathrm{d}V_t &= -\nabla U(X_t)\,\mathrm{d}t - \gamma V_t\,\mathrm{d}t + \sqrt{2\gamma}\,\mathrm{d}W_t,
\end{align}
\end{subequations}
where $U\colon\mathbb{R}^d\to\mathbb{R}$ is a confining potential satisfying
$U(x)\to\infty$ as $|x|\to\infty$
and $\int_{\RR^d}e^{-U(x)}\ud x<\infty$. Throughout the paper, we assume that
$U\in C^2(\mathbb{R}^d)$ and, for some $K\geq 0$,
\begin{equation}\label{eq:hessian-lower}
  \nabla^2 U(x)\geq -K\,\mathrm{Id},\qquad x\in\mathbb{R}^d.
\end{equation}
Here $X_t$ and $V_t$ denote the position and velocity of the particle
(the mass is set to unity for notational convenience).
The drift term $-\gamma V_t\,\mathrm{d}t$ models viscous damping with friction
coefficient $\gamma>0$, while the noise term $\sqrt{2\gamma}\,\mathrm{d}W_t$,
driven by a $d$-dimensional Wiener process $W_t$, represents random forcing
from the thermal environment. The specific ratio of these two coefficients is
dictated by the fluctuation--dissipation relation, which ensures that the Gibbs
measure is preserved.

The density $\varrho(t,x,v)$ of the law of the Langevin dynamics~\eqref{eq:langevin}
satisfies the kinetic Fokker--Planck equation
\begin{equation}\label{eq:kineticfokkerplanck}
  \partial_t\varrho
  = \bigl(-v\cdot\nabla_x + \nabla_x U\cdot\nabla_v\bigr)\varrho
    + \gamma\bigl(\nabla_v\cdot(v\varrho)+\Delta_v\varrho\bigr).
\end{equation}
Introducing the Hamiltonian $H(x,v)=\tfrac{1}{2}|v|^2+U(x)$, one verifies
directly that the Gibbs density
\begin{equation}\label{eq:gibbs}
  \mu(x,v) = \frac{1}{Z}\,e^{-H(x,v)},
  \qquad
  Z = \int_{\RR^{2d}} e^{-H(x,v)}\ud x\ud v
\end{equation}
is a stationary solution. Under mild assumptions on $U$ (see~\cite{Pavliotis2014}), $\mu$ is in fact the
unique invariant measure of~\eqref{eq:langevin}; by a slight abuse of notation,
we also write $\mathrm{d}\mu = \mu(x,v)\,\mathrm{d}x\,\mathrm{d}v$ for this measure.
Since $H$ separates in $x$ and $v$, $\mu$ is a product measure:
$\mathrm{d}\mu(x,v)=\mathrm{d}\mu_x(x)\,\mathrm{d}\kappa(v)$, where
$\mathrm{d}\mu_x\propto e^{-U(x)}\,\mathrm{d}x$ is the position marginal and
$\mathrm{d}\kappa$ is the standard Gaussian on $\mathbb{R}^d$.

In this note, we only consider real functions. For a probability measure $\pi$, define the weighted $L^2$ inner product
\begin{equation}
  \langle f, g\rangle_{L^2(\pi)} = \int fg\,\mathrm{d}\pi,
\end{equation}
and let $L^2(\pi)$ denote the associated Hilbert space. We write $L^2_0(\pi) = \{f \in L^2(\pi) \mid \int f\,\mathrm{d}\pi = 0\}$ for the closed subspace of mean-zero functions.

The ergodic properties of~\eqref{eq:langevin} are conveniently studied through the backward Kolmogorov equation:
\begin{equation}\label{eq:kolmogorov}
  \partial_t f = \mathcal{L}f
  := \bigl(v\cdot\nabla_x - \nabla_x U\cdot\nabla_v\bigr)f
     + \gamma\bigl(-v\cdot\nabla_v+\Delta_v\bigr)f.
\end{equation}
Here and throughout, we consider $\mc L$ as the closure in $L^2(\mu)$ of $(\mc{L}, C_c^\infty(\RR^{2d}))$, with domain denoted by $D(\mc L)$.  By \cite{GS16}*{Theorem~3.5}, $(\mc L, D(\mc L))$ is maximally dissipative and generates a strongly continuous contraction semigroup on $L^2(\mu)$.  Let $\mc L^\dagger$ denote the formal adjoint of $\mc L$ with respect to the Lebesgue measure; it is precisely the operator on the right-hand side of~\eqref{eq:kineticfokkerplanck}, so that the Fokker--Planck equation reads $\partial_t\varrho=\mc L^\dagger\varrho$.
The invariance of $\mu$ gives $\mc L^\dagger\mu=0$, and hence $\langle\mc L\phi,1\rangle_{L^2(\mu)}=0$ for $\phi\in D(\mc L)$. It follows that $1\in D(\mc L^*)$ and $\mc L^*1=0$, where $\mc L^*$ is the Hilbert-space adjoint of $\mc L$ in $L^2(\mu)$.
Hence the semigroup $\mc{P}_t = e^{t\mc L}$ preserves the mean and leaves $L^2_0(\mu)$ invariant. Its qualitative convergence to equilibrium in $L^2(\mu)$ is classical under standard confining assumptions: for every $f_0\in L^2(\mu)$ and $f(t) = \mc{P}_t f_0$, one has
\begin{equation*}
  \big\|\mc{P}_t f_0 -\int_{\RR^{2d}} f_0 \ud\mu\big\|_{L^2(\mu)} \to 0
  \qquad\text{as } t\to\infty.
\end{equation*}
See, for instance, \cites{Villani09}.
In this note, we aim to establish a quantitative decay rate with explicit dependence on the friction coefficient and the Poincar\'e constant (\cref{thm:main} and Corollary~\ref{cor:optimized-shift}).

The generator decomposes as
\begin{subequations}
    \begin{align}
    \mc{L} &= \mc{L}_a + \gamma \mc{L}_s,  \label{eq:Ldecomp1}\\
    \text{where} \quad
    \mc{L}_a &:= v \cdot \nabla_x - \nabla_x U \cdot \nabla_v, \quad \text{and} \quad
    \mc{L}_s := -v \cdot \nabla_v + \Delta_v.  \label{eq:Ldecomp2}
    \end{align}
\end{subequations}
Here $\mathcal{L}_a$ is the Liouville operator associated with the conservative
Hamiltonian flow, expressed via the Poisson bracket as
\begin{equation*}
      \mathcal{L}_a f = \{f, H\} := \nabla_v H\cdot\nabla_x f - \nabla_x H\cdot\nabla_v f.
\end{equation*}
The operator $\mc{L}_s$ models the fluctuation and dissipation induced by the environment. For $f,g\in C_c^\infty(\RR^{2d})$, one verifies that $\mc{L}_a$ is antisymmetric and $\mc{L}_s$ is symmetric in $L^2(\mu)$:
\begin{equation}\label{eq:symmetryantisymmetric}
    \average{f, \mc{L}_a g}_{L^2(\mu)} = - \average{\mc{L}_a f, g}_{L^2(\mu)} \qquad \text{and} \qquad \average{f, \mc{L}_s g}_{L^2(\mu)} = \average{\mc{L}_s f, g}_{L^2(\mu)}.
\end{equation}

Note that $\mc L_s$ has precisely the functions independent of $v$ in its kernel:
\begin{equation}\label{eq:kerLs}
    \ker(\mc{L}_s) = \left\{ f \in L^2(\mu) \;\big|\; \nabla_v f(x, v) = 0 \right\},
\end{equation}
and hence $\ker(\mc{L}_s)$ is much larger than $\ker(\mc{L})$. Because dissipation occurs only in the velocity variable, the full Langevin generator $\mc{L}$ is not coercive; however, it is hypocoercive (see \cite{Villani09} for the precise definition).
We assume that $\mu_x$ satisfies a Poincar\'e inequality with constant $m > 0$:
\begin{equation}\label{eq:poincare}
    \norm{g}_{L^2(\mu_x)}^2 \le \frac{1}{m}\norm{\nabla_x g}_{L^2(\mu_x)}^2
    \qquad\text{for all } g \in L^2_0(\mu_x) \text{ with } \nabla_x g \in L^2(\mu_x).
\end{equation}

Our main result is the following explicit hypocoercive $L^2$-decay estimate for the underdamped Langevin dynamics~\eqref{eq:langevin}. The proof is given in Section~\ref{sec:proof}.

\begin{theorem}\label{thm:main}
Suppose that $U \in C^2(\mathbb{R}^d)$ with $U(x) \to \infty$ as $|x| \to \infty$ and $\int_{\RR^d}e^{-U(x)}\ud x<\infty$.  Let $\mu(\mathrm{d}x\,\mathrm{d}v) \propto e^{-U(x)-|v|^2/2}\,\mathrm{d}x\,\mathrm{d}v$ be the Gibbs measure on $\mathbb{R}^{2d}$. Assume:
\begin{enumerate}
  \item[\textup{(i)}] the position marginal $\mu_x$ satisfies the
  Poincar\'e inequality~\eqref{eq:poincare} with constant $m > 0$;
  \item[\textup{(ii)}] the Hessian lower bound~\eqref{eq:hessian-lower} holds
  with $K \ge 0$.
\end{enumerate}
Let $f(t) = e^{t\mathcal{L}}f_0$ be the solution to~\eqref{eq:kolmogorov}
with $f_0 \in L^2_0(\mu)$. For any fixed friction coefficient $\gamma > 0$ and shift $\alpha\ge0$,
the following decay estimate holds for all $t \ge 0$:
\begin{equation}\label{eq:general-decay}
 \|f(t)\|_{L^2(\mu)}
  \le
  \left(\frac{1+q_{\alpha,\gamma}}{1-q_{\alpha,\gamma}}\right)^{1/2}
  \exp\!\left(-\Lambda_{\alpha,\gamma}t\right)
  \|f_0\|_{L^2(\mu)},
\end{equation}
where the factor $q_{\alpha,\gamma}$ and the decay rate $\Lambda_{\alpha,\gamma}$ are given by
\begin{equation}\label{eq:q-lambda}
  q_{\alpha,\gamma}
  :=\frac{4s_\alpha\gamma c_\alpha}
  {4s_\alpha+(\gamma c_\alpha+C_\alpha)^2} < 1, \qquad
  \Lambda_{\alpha,\gamma}
  :=\frac{s_\alpha^2\gamma}
  {4s_\alpha+(\gamma c_\alpha+C_\alpha)^2
  +4s_\alpha\gamma c_\alpha}.
\end{equation}
Here the constants $s_\alpha$, $c_\alpha$, and $C_\alpha$ with $\alpha \ge 0$ are defined by
\begin{equation}\label{eq:main-constants}
 s_\alpha:=\frac{m}{m+\alpha},\qquad
  c_\alpha:=\sup_{\lambda\ge m}\frac{\sqrt{\lambda}}{\alpha+\lambda}
  =
  \begin{cases}
    \dfrac{\sqrt m}{m+\alpha}, & 0\le\alpha\le m,\\[6pt]
    \dfrac{1}{2\sqrt\alpha}, & \alpha\ge m,
  \end{cases}
  \qquad
  C_\alpha:=\bigl(2+2Kc_\alpha^2\bigr)^{1/2}.
\end{equation}
\end{theorem}

\begin{coro}[Optimized convergence]\label{cor:optimized-shift}
Under the assumptions of \thmref{thm:main}, for every fixed $\gamma>0$, the rate $\Lambda_{\alpha,\gamma}$ in \eqref{eq:q-lambda} is strictly decreasing in $\alpha \ge 0$ and thus
\begin{equation*}
  \sup_{\alpha>0}\Lambda_{\alpha,\gamma} =\lim_{\alpha\downarrow0}\Lambda_{\alpha,\gamma} =\Lambda_{0,\gamma}.
\end{equation*}
In particular, when $\alpha = 0$, we have
\begin{equation}\label{eq:shift-optimized-rate}
  \Lambda_{0,\gamma}
  =
  \frac{\gamma}
  {6+\frac{2K}{m}
    +\frac{2\gamma}{\sqrt m}\left(2+\sqrt{2+\frac{2K}{m}}\right)
    +\frac{\gamma^2}{m}},
\end{equation}
which has the following scaling behaviors:
\begin{equation}\label{eq:shift-optimized-asymptotics}
  \Lambda_{0,\gamma}
  = \frac{\gamma}{6+\frac{2K}{m}} + \mathcal{O}(\gamma^2)
  \quad(\text{as }\ \gamma\to0),
  \qquad
  \Lambda_{0,\gamma}
  = \frac{m}{\gamma} + \mathcal{O}(\gamma^{-2})
  \quad(\text{as }\ \gamma\to\infty).
\end{equation}
Moreover, the rate $ \Lambda_{0,\gamma}$ is maximized at $\gamma_*=\sqrt{6m+2K}$, and
\begin{equation}\label{eq:fully-optimized-rate}
 \sup_{\gamma > 0} \Lambda_{0,\gamma} =  \Lambda_{0,\gamma_*} = \frac{\sqrt m} {2\left(2+\sqrt{2+\frac{2K}{m}}+\sqrt{6+\frac{2K}{m}}\right)}.
\end{equation}
\end{coro}
When $U$ is convex, i.e., $K=0$, \eqref{eq:shift-optimized-rate} becomes
\begin{equation} \label{eq:convexrate}
  \Lambda_{0,\gamma} = \frac{\gamma}
  {6+2(2+\sqrt2)\gamma/\sqrt m+\gamma^2/m}.
\end{equation}
In this case, the optimal friction and convergence rate are given by
\begin{equation*}
   \gamma_*=\sqrt{6m},
  \qquad
  \Lambda_{0,\gamma_*} = \frac{\sqrt m}{2(2+\sqrt2+\sqrt6)}.
\end{equation*}
Moreover, we have, from \eqref{eq:convexrate},
\begin{equation*}
   \Lambda_{0,\gamma}=\mathcal{O}(\gamma)\quad(\gamma\ll\sqrt m),
  \qquad
  \Lambda_{0,\gamma}=\mathcal{O}(\sqrt m)\quad(\gamma\asymp\sqrt m),
  \qquad
  \Lambda_{0,\gamma}=\mathcal{O}(m/\gamma)\quad(\gamma\gg\sqrt m).
\end{equation*}

The quantitative estimate in \thmref{thm:main} is not new: up to absolute constants, it agrees with that of~\cite{Cao2023}, which was obtained via a space--time Poincar\'{e} inequality and shown to be optimal in the sense discussed there. Specifically, we see from \eqref{eq:shift-optimized-rate} that, for any $\gamma > 0$,
\begin{equation*}
    \frac{1}{6}\,\frac{m\gamma}{(\gamma+\sqrt{m}+\sqrt{K})^2}
    \;\le\; \Lambda_{0,\gamma} \;\le\;
    \frac{m\gamma}{(\gamma+\sqrt{m}+\sqrt{K})^2},
\end{equation*}
which matches exactly the scaling in $m$, $K$, and $\gamma$ of \cite{Cao2023}*{(13)}, across the low-, critical-, and high-friction regimes. The purpose of this note is to show that the same optimal hypocoercive estimate can be recovered using the modified $L^2$ method~\cites{DMS09, DMS15}; to the best of our knowledge, this has not been recorded in the literature.

\subsection*{Related work} The modified $L^2$ hypocoercivity method was introduced by Dolbeault--Mouhot--Schmeiser~\cite{DMS09} for linear-relaxation models and subsequently developed into an abstract framework for mass-conserving linear kinetic equations in~\cite{DMS15}. The core idea traces back to the earlier work of H\'{e}rau~\cite{herau2006hypocoercivity}, dealing with the linear relaxation Boltzmann equation under stronger assumptions on the confining potential. For Langevin dynamics, Grothaus--Stilgenbauer~\cite{GS16} extended the DMS approach to give a rigorous $L^2$ hypocoercivity treatment of the backward Kolmogorov semigroup, addressing the relevant operator-domain issues and obtaining an explicit exponential convergence rate in the friction coefficient. Roussel--Stoltz~\cite{RS18} used hypocoercivity to analyze spectral Galerkin approximations of Poisson equations associated with the Langevin generator, while Bernard et al.~\cite{BFLS22} developed a Schur-complement approach to directly bound the inverse of hypocoercive generators, with applications to Langevin-type dynamics.

Building on the space--time variational approach to kinetic Fokker--Planck equations developed by Albritton et al.~\cite{AAMN24}, Cao--Lu--Wang~\cite{Cao2023} derived an explicit all-friction estimate for underdamped Langevin dynamics, and, for convex potentials satisfying~\eqref{eq:poincare} under regularity assumptions, obtained the optimal $\mathcal{O}(\sqrt{m})$ rate after optimizing the
friction coefficient. Brigati--Stoltz~\cite{BrigatiStoltz25} subsequently developed a fully constructive space--time Poincar\'e--Lions approach for kinetic Fokker--Planck equations with more general confining and kinetic energies. For Gaussian velocities, \cite{BrigatiStoltz25}*{Theorem~4} gives a lower bound on the time-averaged decay rate proportional to $\gamma/(1+\gamma^2)$,
with a prefactor depending on the potential $U$ and the time window. This scales as $\gamma$ at low friction and $\gamma^{-1}$ at high friction, matching our scalings \eqref{eq:shift-optimized-asymptotics}, though without isolating the dependence on $m$, $K$, and $\gamma$ in closed form. Brigati et al.~\cite{BrigatiStoltzWang24} treated weakly confining spatial potentials whose position marginals satisfy a weighted (rather than standard) Poincar\'e inequality, and allowed velocity measures satisfying only a weak Poincar\'e inequality.
In this case, they obtained explicit stretched-exponential or algebraic $L^2$ decay rates for bounded initial data.

More recently, second-order lifts have linked hypocoercivity to the diffusive-to-ballistic acceleration of reversible dynamics. Eberle--L\"orler~\cite{EberleLoerler24} proved that such quadratic acceleration is optimal and recast space--time hypocoercivity in the
language of lifts. Brigati--L\"orler--Wang~\cite{BrigatiLoerlerWang25} unified this viewpoint with variational hypocoercivity and recovered the $\sqrt{m}$ decay-rate scale for standard Langevin dynamics. Related
developments introduced flow Poincar\'e inequalities~\cite{EGHLM25} and extended lifting to quantum Markov semigroups~\cite{li2025speeding}. The present note shows that the optimal rate for the underdamped Langevin dynamics is also accessible within the DMS $L^2$ framework, without recourse to space--time Poincar\'{e} or lifting machinery.

\section{Hypocoercivity estimate \`a la Dolbeault--Mouhot--Schmeiser}\label{sec:dmsreview}

Denoting by $\nabla_x^*$ and $\nabla_v^*$ the formal adjoints of $\nabla_x$ and $\nabla_v$ in $L^2(\mu)$, we have
\begin{equation}
    \nabla_x^* = -\nabla_x + \nabla_x U, \qquad \nabla_v^* = -\nabla_v + v,
\end{equation}
so that $\mathcal{L}_s = -\nabla_v^*\nabla_v$. We also introduce the overdamped Langevin generator
\begin{equation} \label{def:overdampedlgv}
  \mc{L}_{\mathrm{o}} = -\nabla_x^*\nabla_x = \Delta_x - \nabla_x U\cdot\nabla_x.
\end{equation}
The operator $(\mc L_{\mathrm o},C_c^\infty(\RR^d))$ is essentially self-adjoint in $L^2(\mu_x)$ \cite{BKR97}; we denote its nonpositive self-adjoint closure by $\mc L_{\mathrm o}$. Its associated closed Dirichlet form is
\begin{equation*}
    \mathcal E_{\mathrm o}(f,g) := \langle\nabla_x f,\nabla_x g\rangle_{L^2(\mu_x)},
\end{equation*}
where
\begin{equation*}
  D(\mathcal E_{\mathrm o})
  :=\overline{C_c^\infty(\RR^d)}
  ^{\;\|\cdot\|_{L^2(\mu_x)}+
     \|\nabla_x\cdot\|_{L^2(\mu_x)}}
  =D\bigl((-\mc L_{\mathrm o})^{1/2}\bigr).
\end{equation*}
In particular, $\ker\mc L_{\mathrm o}$ consists of constant functions, and the Poincar\'e inequality~\eqref{eq:poincare} gives
\begin{equation}\label{eq:overdamped-gap}
  \operatorname{spec}
  \bigl(-\mc L_{\mathrm o}|_{L^2_0(\mu_x)}\bigr)\subset[m,\infty).
\end{equation}

We denote by $\Pi_v$ the orthogonal projection onto $\ker\mathcal{L}_s$, given by
\begin{equation} \label{def:piv}
  (\Pi_v f)(x) = \int_{\mathbb{R}^d} f(x,v)\,\mathrm{d}\kappa(v).
\end{equation}
For a function $f(x,v)$, one has that $\Pi_v f$ is independent of $v$ and $\mc L_a\Pi_vf=v\cdot\nabla_x\Pi_vf$. We can regard $\mc L_a\Pi_v$ as the densely defined closed operator with domain
\begin{equation*}
   D(\mc L_a\Pi_v) :=\bigl\{f\in L^2(\mu):\Pi_vf\text{ has a weak }x\text{-gradient in }  L^2(\mu_x)\bigr\},
\end{equation*}
since the weak gradient is closed and $\|v\cdot g(x)\|_{L^2(\mu)}  =\|g\|_{L^2(\mu_x)}$. Then, the antisymmetry of $\mc L_a$ and $\int v\ud\kappa(v) = 0$ give, for any $\phi\in C_c^\infty(\RR^{2d})$,
\begin{subequations}
\begin{align}
    &(\mathcal{L}_a\Pi_v)^* \phi = -\Pi_v\mathcal{L}_a \phi,  \label{eq:ops1} \\
    &  \Pi_v\mathcal{L}_a\Pi_v \phi  = 0\,. \label{eq:ops2}
\end{align}
\end{subequations}
Moreover, recalling~\eqref{def:overdampedlgv}, for $f,g \in C_c^\infty(\RR^d)$, we have
\begin{equation} \label{eq:liftstru}
    \langle f, -\mc{L}_{\mathrm{o}}\, g\rangle_{L^2(\mu_x)}
    = \langle \nabla_x f, \nabla_x g\rangle_{L^2(\mu_x)}
    = \langle \mathcal{L}_a f, \mathcal{L}_a g\rangle_{L^2(\mu)},
\end{equation}
where the last equality follows from $\int_{\RR^{d}} v_i v_j\,\mathrm{d}\kappa(v) = \delta_{ij}$. Here and throughout, we identify $f(x)$ with its trivial lift $f(x,v) := f(x)$. By closure, the identity \eqref{eq:liftstru} extends to
\begin{align*}
   \langle\mc L_a\Pi_vf,\mc L_a\Pi_vg\rangle_{L^2(\mu)} =\mathcal E_{\mathrm o}(\Pi_vf,\Pi_vg), \qquad f,g\in D(\mc L_a\Pi_v).
\end{align*}
Then, by uniqueness of the self-adjoint operator associated with a closed nonnegative form, we have the operator form of \eqref{eq:liftstru}:
\begin{equation}\label{eq:liftstru2}
  (\mc L_a\Pi_v)^*\mc L_a\Pi_v = - \mc L_{\mathrm o}\Pi_v.
\end{equation}

To motivate the modified $L^2$ functional of \cites{DMS09,DMS15}, we first examine the decay of the standard $L^2(\mu)$ norm along \eqref{eq:kolmogorov}:
\begin{equation}
\partial_t \Bigl(\frac{1}{2}\|f\|_{L^2(\mu)}^2\Bigr)
= \langle f, \mathcal{L}f\rangle_{L^2(\mu)}
= \gamma\langle f, \mathcal{L}_s f\rangle_{L^2(\mu)},
\end{equation}
where the last equality uses the antisymmetry of $\mathcal{L}_a$. Since the right-hand side vanishes for any $f\in\ker\mathcal{L}_s$, the $L^2(\mu)$ norm alone does not yield coercive decay.

To recover indirect dissipation from $\mc{L}_a$, the idea of Dolbeault--Mouhot--Schmeiser \cites{DMS09,DMS15} is to introduce a modified $L^2$ functional as a Lyapunov function:
\begin{equation*}
  \mf{L}(f) := \frac{1}{2}\|f\|_{L^2(\mu)}^2 - \varepsilon\langle \mathcal{A}f, f\rangle_{L^2(\mu)},
\end{equation*}
where the corrector $\mathcal{A}$ is given by
\begin{equation*}
  \mathcal{A} := (1 + (\mc{L}_a \Pi_v)^*\mc{L}_a \Pi_v)^{-1}(\mathcal{L}_a\Pi_v)^* = (1 - \mc{L}_{\mathrm{o}})^{-1}(\mathcal{L}_a\Pi_v)^*,
\end{equation*}
where the second equality uses $(\mc{L}_a\Pi_v)^*\mc{L}_a\Pi_v = -\mc{L}_{\mathrm{o}}\Pi_v$ (see~\eqref{eq:liftstru2} above) together with the fact that the range of $(\mc{L}_a\Pi_v)^*$ lies in $\operatorname{Ran}\Pi_v$. See~\cites{GS16,RS18,Cao2023} for adaptations of the DMS approach to the Langevin equation~\eqref{eq:langevin}. One can show that the Lyapunov function $\mf{L}(f)$ is equivalent to $\|f\|_{L^2(\mu)}^2$ and thus yields quantitative convergence bounds. However, as shown in~\cite{Cao2023}*{Proposition~B.2}, when $U$ is convex and the friction is in the critical regime $\gamma \sim \sqrt{m}$, this choice of corrector yields a convergence rate of only $\mathcal{O}(m^{5/2})$ as $m \to 0$, failing to achieve the optimal $\mathcal{O}(\sqrt{m})$ rate.

Our key observation is that a simple modification of $\mc A$ yields the optimal convergence-rate scaling within essentially the same framework. In general, for $\alpha \ge 0$, we define the shifted corrector, initially on
$D((\mc L_a\Pi_v)^*)\cap L^2_0(\mu)$, by
\begin{equation}\label{eq:defAalpha}
  \mc A_\alpha :=(\alpha-\mc L_{\mathrm o})^{-1}(\mc L_a\Pi_v)^*.
\end{equation}
Here, for $\alpha>0$, the operator $\alpha-\mc L_{\mathrm o}$ is strictly positive and has a bounded inverse, while for $\alpha=0$ the inverse is understood as the reduced inverse of $-\mc L_{\mathrm o}$ on $L^2_0(\mu_x)$.  Indeed, noting $\mc L_a\Pi_v=(\mc L_a\Pi_v)\Pi_v$ and $\mc L_a\Pi_v1=0$, taking adjoints shows
\begin{equation} \label{eq:relarange}
  \operatorname{Ran}(\mc L_a\Pi_v)^* \subset \operatorname{Ran}\Pi_v\cap L^2_0(\mu) \cong L^2_0(\mu_x),
\end{equation}
and then the spectral gap \eqref{eq:overdamped-gap} implies that the restriction of $-\mc L_{\mathrm o}$ to $L^2_0(\mu_x)$ is invertible, with bounded inverse. Thus, $\mc A_\alpha$ is well defined for all $\alpha \ge 0$. In particular, by \eqref{eq:liftstru2}, the polar decomposition of $(\mc L_a\Pi_v)^*$, and the spectral calculus, $\mc A_\alpha$ extends uniquely to a bounded operator on $L^2_0(\mu)$, and
\begin{equation}\label{eq:boundAalpha}
  \|\mc A_\alpha\|_{L^2_0(\mu)\to L^2_0(\mu)}
  =
  \|\mc A_\alpha\mc A_\alpha^*\|_
    {L^2_0(\mu)\to L^2_0(\mu)}^{1/2}
  \le
  \sup_{\lambda\ge m}\frac{\sqrt{\lambda}}{\alpha+\lambda}
  =c_\alpha<\infty;
\end{equation}
see \eqref{eq:main-constants} for the explicit expression for $c_\alpha$. Then, for $\varepsilon>0$ and $f\in L^2_0(\mu)$, define the corresponding modified DMS $L^2$-functional:
\begin{equation}\label{eq:defHalpha}
  \mf L_{\alpha,\varepsilon}(f)
  :=\frac12\|f\|_{L^2(\mu)}^2
  -\varepsilon\langle\mc A_\alpha f,f\rangle_{L^2(\mu)}.
\end{equation}
It is worth pointing out that variants of the DMS auxiliary operator with a positive resolvent shift have recently been employed in analyses of adaptive Langevin dynamics~\cites{LSS20,delande2025sharp} and velocity-jump processes~\cite{MRZ23}. In this work, we consider the family of correctors $\{\mc{A}_\alpha\}_{\alpha \ge 0}$ for the underdamped Langevin generator and optimize the convergence rate over $\alpha$.




The mechanism behind the shift may be easiest to see on the ``slow part'' $f_S := \Pi_v f \in \ker \mathcal{L}_s$. This component is not directly damped by $\mathcal{L}_s$, so the corrector must recover coercivity through the overdamped operator $-\mathcal{L}_{\mathrm{o}}$. As will be shown in \eqref{eq:basic-identities3} in \cref{sec:proof}, we have
\begin{equation*}
      \mathcal{A}_\alpha \mathcal{L}_a \Pi_v = (\alpha - \mathcal{L}_{\mathrm{o}})^{-1}(-\mathcal{L}_{\mathrm{o}})\Pi_v.
\end{equation*}
Thus, if $f_S$ is an eigenfunction of $-\mathcal{L}_{\mathrm{o}}$ with eigenvalue $\lambda$, the corrector produces a prefactor $\lambda/(\alpha + \lambda)$. Since the Poincar\'{e} inequality gives $\lambda \ge m$ on non-constant eigenfunctions, this factor is always at least $1/2$ whenever $\alpha \le m$: the shifted corrector then turns
even the slowest macroscopic mode into an order-one coercive contribution; see \eqref{eq:slow-coercive} for the precise estimate. By contrast, when $\alpha$ is an order-one constant independent of $m$, the same factor $\lambda/(\alpha + \lambda)$ is only of order $m$ on the slowest mode where $\lambda \sim m$. The macroscopic coercive contribution is then too weak, and one loses a power of $m$ in the final rate.

Compared with approaches based on a space--time Poincar\'{e} inequality \cites{AAMN24,Cao2023} and with later lifting-based approaches \cites{EberleLoerler24,BrigatiLoerlerWang25,EGHLM25,li2025speeding}, the present analysis is arguably simpler: it operates on a single time slice, rather than requiring integration over a time interval for the hypocoercive dissipation to accumulate.

\section{Proofs of \thmref{thm:main} and \corref{cor:optimized-shift}}\label{sec:proof}

We begin by recalling two standard estimates. First, integrating the Gaussian Poincar\'e inequality in the velocity
variable against $\mu_x$ gives
\begin{equation}\label{eq:gauss}
  \|(1-\Pi_v)f\|_{L^2(\mu)}^2
  \le \|\nabla_v f\|_{L^2(\mu)}^2,
\end{equation}
for every $f\in L^2(\mu)$ with $\nabla_v f\in L^2(\mu;\RR^d)$.
Second, the Bochner identity \cite{BGL14} gives

\begin{equation*}
    \frac12 \mc L_{\mathrm o} |\nabla_x h|^2 - \langle \nabla_x h,\nabla_x \mc L_{\mathrm o} h\rangle
= \|\nabla_x^2 h\|_{\mathrm{HS}}^2
+ \langle \nabla^2 U\,\nabla_x h,\nabla_x h\rangle\,, \qquad \text{for }\ h\in C_c^\infty(\RR^d).
\end{equation*}
Integrating this identity with respect to $\mu_x$ and using, from \eqref{eq:liftstru},
\begin{equation*}
    \int \mc L_{\mathrm o} |\nabla_x h|^2 \ud \mu_x=0,
\qquad
-\int \langle \nabla_x h,\nabla_x \mc L_{\mathrm o} h\rangle \ud \mu_x = \int (\mc L_{\mathrm o} h)^2 \ud \mu_x,
\end{equation*}
we obtain
\begin{equation}\label{eq:bochner-identity}
  \|\mathcal{L}_{\mathrm{o}} h\|_{L^2(\mu_x)}^2
  = \|\nabla_x^2 h\|_{L^2(\mu_x)}^2
    + \int_{\RR^{d}} \nabla^2 U\,\nabla_x h\cdot\nabla_x h\,\mathrm{d}\mu_x,
\end{equation}
which, combined with the Hessian lower bound \eqref{eq:hessian-lower}, yields
\begin{equation}\label{eq:bochner}
  \|\nabla_x^2 h\|_{L^2(\mu_x)}^2
  \le \|\mathcal{L}_{\mathrm{o}} h\|_{L^2(\mu_x)}^2 + K\|\nabla_x h\|_{L^2(\mu_x)}^2.
\end{equation}
Finally, by \eqref{eq:relarange}, the corrector $\mc A_\alpha$ defined in \eqref{eq:defAalpha} satisfies
\begin{subequations}\label{eq:basic-identities-alpha}
\begin{align}
  \mc A_\alpha &= \Pi_v\mc A_\alpha, \label{eq:basic-identities1}\\
  \mc A_\alpha\Pi_v &=0, \label{eq:basic-identities2}\\
  \mc A_\alpha\mc L_a\Pi_v
  &=(\alpha-\mc L_{\mathrm o})^{-1}(-\mc L_{\mathrm o})\Pi_v,
  \label{eq:basic-identities3}
\end{align}
\end{subequations}
where \eqref{eq:basic-identities2} follows from \eqref{eq:ops2}, while \eqref{eq:basic-identities3} follows from \eqref{eq:liftstru2}. More precisely, the first two identities hold on $L^2_0(\mu)$ because $\mc A_\alpha$ is bounded. The identity \eqref{eq:basic-identities3} is initially understood on $D(\mc L_a\Pi_v)\cap L^2_0(\mu)$, while it extends to $L^2_0(\mu)$ since its right-hand side defines a bounded operator on $L^2_0(\mu)$.

\subsection*{Bounds on the corrector}
The following lemma records the basic bounds on the operators $\mc{A}_\alpha$, $\mc{L}_a\mc{A}_\alpha$, and $\mc{A}_\alpha\mc{L}_a(1-\Pi_v)$.

\begin{lemma}\label{lem:A-bounds}
For $\alpha\ge0$, let $c_\alpha$ and $C_\alpha$ be the constants given in \eqref{eq:main-constants}. Then, for any $\phi\in C_c^\infty(\RR^{2d})\cap L^2_0(\mu)$, we have
\begin{align}
  \|\mc A_\alpha\phi\|_{L^2(\mu)}
  &\le c_\alpha \|\phi\|_{L^2(\mu)}, \label{eq:A-bound}\\
  \|\mc L_a\mc A_\alpha\phi\|_{L^2(\mu)}
  &\le \|\phi\|_{L^2(\mu)}, \label{eq:A1}\\
  \|\mc A_\alpha\mc L_a(1-\Pi_v)\phi\|_{L^2(\mu)}
  &\le C_\alpha\|(1-\Pi_v)\phi\|_{L^2(\mu)}. \label{eq:A2}
\end{align}
In particular, $\mathcal{A}_\alpha$, $\mathcal{L}_a\mathcal{A}_\alpha$, and $\mathcal{A}_\alpha\mathcal{L}_a(1-\Pi_v)$ extend uniquely to bounded operators on $L^2_0(\mu)$ with the same bounds.
Moreover, if $2\varepsilon c_\alpha<1$, then
\begin{equation}\label{eq:H-equiv}
  \frac{1-2\varepsilon c_\alpha}{2}\|\phi\|_{L^2(\mu)}^2
  \le \mf L_{\alpha,\varepsilon}(\phi) \le
  \frac{1+2\varepsilon c_\alpha}{2}\|\phi\|_{L^2(\mu)}^2
  \qquad \forall \phi\in L^2_0(\mu).
\end{equation}
\end{lemma}

\begin{proof}
The first bound \eqref{eq:A-bound} is proved in \eqref{eq:boundAalpha}. For \eqref{eq:A1}, let $T:=(\mc L_a\Pi_v)^*$, regarded as a closed operator from $D(T) \subset L^2_0(\mu)$ into $\operatorname{Ran}\Pi_v\cap L^2_0(\mu)\cong L^2_0(\mu_x)$. By \eqref{eq:liftstru2}, $TT^*=-\mc L_{\mathrm o}$  on $L^2_0(\mu_x)$, with domain $D(\mc L_{\mathrm o})\cap L^2_0(\mu_x)$. Then, we have the bounded operator $\mc A_\alpha=(\alpha+TT^*)^{-1}T$ for $\alpha \ge 0$, recalling \eqref{eq:defAalpha}. For $f \in D(T)$, we have $\mc A_\alpha f \in D(TT^*)\subset D(T^*)$, and hence
\begin{equation*}
   \mc L_a\mc A_\alpha f =T^*(\alpha+TT^*)^{-1}Tf
\end{equation*}
is well defined. Let $T=V|T|$ with $|T|:=(T^*T)^{1/2}$ be the polar decomposition of $T$. Then, we can compute for $\alpha > 0$, on $(\ker T)^\perp$,
\begin{align*}
  T^*(\alpha+TT^*)^{-1}T
  &=
  |T|V^*(\alpha+TT^*)^{-1}V|T|\\
  &=
  |T|(\alpha+|T|^2)^{-1}|T|  = (T^*T)(\alpha+T^*T)^{-1},
\end{align*}
by the intertwining relation $(\alpha+TT^*)^{-1}V = V(\alpha+T^*T)^{-1}$. In particular, at $\alpha = 0$, the same polar-decomposition calculation gives
\begin{equation*}
    T^*(TT^*)^{-1}T  = V^*V = \mathbf 1_{(0,\infty)}(T^*T).
\end{equation*}
It follows that, for $\alpha \ge 0$, on $D(T)$,
\begin{equation*}
   \mc L_a\mc A_\alpha =  g_\alpha(T^*T),
  \qquad
  g_\alpha(\lambda):=
  \begin{cases}
    \dfrac{\lambda}{\alpha+\lambda}, & \alpha>0,\\[6pt]
    \mathbf 1_{(0,\infty)}(\lambda), & \alpha=0.
  \end{cases}
\end{equation*}
Since $0\le g_\alpha\le1$, this operator extends to a positive contraction on $L^2_0(\mu)$. This proves \eqref{eq:A1}.

To prove~\eqref{eq:A2}, let $B_\alpha=\mc A_\alpha\mc L_a(1-\Pi_v)$ for simplicity, with dense domain $ D(B_\alpha):=C_c^\infty(\RR^{2d})\cap L^2_0(\mu)$. Since $\operatorname{Ran}\mc A_\alpha\subset \operatorname{Ran}\Pi_v\cap L^2_0(\mu)$, the same inclusion holds for $\operatorname{Ran}B_\alpha$. By duality, it suffices to prove that any $g\in\operatorname{Ran}\Pi_v\cap L^2_0(\mu)$ belongs to $D(B_\alpha^*)$, satisfies $\|B_\alpha^*g\|_{L^2(\mu)}\le C_\alpha\|g\|_{L^2(\mu)}$, and has $B_\alpha^*g\in\operatorname{Ran}(1-\Pi_v)$.

Let
\begin{equation} \label{auxeq:b}
    h=(\alpha-\mc L_{\mathrm o})^{-1}g
    \in D(\mc L_{\mathrm o})\cap L^2_0(\mu_x),
\end{equation}
and choose $\widetilde h_n\in C_c^\infty(\RR^d)$ converging to $h$ in the graph norm of $\mc L_{\mathrm o}$. We also let $\chi_x\in C_c^\infty(\RR^d)$ with $\int\chi_x\ud\mu_x=1$ and
\begin{equation*}
    h_n:=\widetilde h_n -\left(\int\widetilde h_n\ud\mu_x\right)\chi_x.
\end{equation*}
It follows that $h_n\in C_c^\infty(\RR^d)\cap L^2_0(\mu_x)$ and
\begin{equation} \label{eq:convergraph}
   h_n\to h, \qquad
  \mc L_{\mathrm o}h_n\to\mc L_{\mathrm o}h \quad\text{in }L^2(\mu_x).
\end{equation}
We now introduce the associated $g_n$ by $g_n=(\alpha-\mc L_{\mathrm o})h_n$, which satisfies $g_n\to g$ in $L^2_0(\mu_x)$. Then, for $\phi\in C_c^\infty(\RR^{2d}) \cap L^2_0(\mu)$, integration by parts gives
\begin{align*}
  \langle B_\alpha\phi,g_n\rangle_{L^2(\mu)}
 =\langle \mc L_a(1-\Pi_v)\phi,\mc L_a h_n\rangle_{L^2(\mu)} = -\langle \phi,(1-\Pi_v)\mc L_a^2h_n\rangle_{L^2(\mu)}.
\end{align*}
Here we used $\mc A_\alpha^*g_n=\mc L_a h_n$ by \eqref{eq:defAalpha}. Hence, $g_n\in D(B_\alpha^*)$ and
\begin{equation*}
    B_\alpha^*g_n =-(1-\Pi_v)\mc L_a^2h_n  =-\sum_{i,j}(v_iv_j-\delta_{ij})\partial_{ij}h_n.
\end{equation*}
From \eqref{eq:bochner} and \eqref{eq:convergraph}, we have
\begin{equation*}
    \|\nabla_x^2 (h_n - h_k) \|_{L^2(\mu_x)}^2 \le \|\mc L_{\mathrm o} (h_n - h_k) \|_{L^2(\mu_x)}^2 + K \|\nabla_x (h_n - h_k) \|_{L^2(\mu_x)}^2 \to 0,
\end{equation*}
where $ \|\nabla_x (h_n - h_k) \|_{L^2(\mu_x)}^2 = \langle h_n - h_k,-\mc L_{\mathrm o} (h_n - h_k)\rangle_{L^2(\mu_x)}\to 0$. This shows that $B_\alpha^*g_n$ converges in $L^2(\mu)$. Since $B_\alpha^*$ is closed and $g_n\to g$, it follows that $g\in D(B_\alpha^*)$ and
\begin{equation*}
    B_\alpha^*g =-\sum_{i,j}(v_iv_j-\delta_{ij})\partial_{ij}h,
\end{equation*}
where $\nabla_x^2h$ denotes the $L^2(\mu_x)$ limit of $\nabla_x^2h_n$. In particular, $B_\alpha^*g\in\operatorname{Ran}(1-\Pi_v)$. From the Gaussian fourth-moment identity
\begin{equation*}
  \int_{\RR^d}(v_iv_j-\delta_{ij})(v_kv_l-\delta_{kl})\ud\kappa
  =\delta_{ik}\delta_{jl}+\delta_{il}\delta_{jk},
\end{equation*}
we then have
\begin{equation*}
    \|B_\alpha^*g\|_{L^2(\mu)}^2 \le2\|\nabla_x^2h\|_{L^2(\mu_x)}^2.
\end{equation*}
This, combined with the inequality  \eqref{eq:bochner} and the identity \eqref{eq:liftstru}, implies
\begin{equation*}
     \|B_\alpha^*g\|_{L^2(\mu)}^2
  \le 2\|\mathcal{L}_{\mathrm{o}} h\|_{L^2(\mu_x)}^2 + 2K \average{h, - \mc{L}_{\mathrm{o}} h}_{L^2(\mu_x)}.
\end{equation*}
Applying the spectral decomposition of $-\mathcal{L}_{\mathrm{o}}$ on
$L^2_0(\mu_x)$ together with \eqref{auxeq:b}, we find
\begin{align} \label{eq:bounda2}
  \|B_\alpha^*g\|_{L^2(\mu)}^2 \le
  \sup_{\lambda\ge m}  \frac{2\lambda^2+2K\lambda}{(\alpha+\lambda)^2}
  \|g\|_{L^2(\mu_x)}^2 \le \bigl(2+2Kc_\alpha^2\bigr)\|g\|_{L^2(\mu_x)}^2
  =C_\alpha^2\|g\|_{L^2(\mu_x)}^2,
\end{align}
where the second inequality is obtained by bounding the two terms in the numerator separately, using $\operatorname{spec}(-\mc{L}_{\mathrm{o}}|_{L^2_0(\mu_x)}) \subset [m,\infty)$ together with
\begin{equation*}
  \frac{\lambda^2}{(\alpha+\lambda)^2}\le1, \qquad \sup_{\lambda\ge m}\frac{\lambda}{(\alpha+\lambda)^2} =c_\alpha^2.
\end{equation*}
Together with $B_\alpha^*g\in\operatorname{Ran}(1-\Pi_v)$, the bound \eqref{eq:bounda2} above gives $\|B_\alpha\phi\|_{L^2(\mu)}\le C_\alpha\|(1-\Pi_v)\phi\|_{L^2(\mu)}$ for $\phi\in D(B_\alpha)$ by duality. The bound extends to $L^2_0(\mu)$ by density. We have completed the proof of \eqref{eq:A2}. Finally, \eqref{eq:H-equiv} follows from \eqref{eq:A-bound} and the Cauchy--Schwarz inequality.
\end{proof}

 \subsection*{Proof of \thmref{thm:main}} We fix $\alpha\ge0$ and $\gamma>0$, and let $s_\alpha,c_\alpha,C_\alpha$ be the constants in \eqref{eq:main-constants}. Let $f_0\in D(\mc L)\cap L^2_0(\mu)$ and $f(t)=e^{t\mc L}f_0$.  Then, with $D(\mc L)$ equipped with its graph norm, we have
\begin{equation*}
      f(t)\in C([0,\infty);D(\mc{L}))\cap C^1([0,\infty);L^2_0(\mu))\,.
\end{equation*}
Differentiating \eqref{eq:defHalpha} along $f(t)$ gives
\begin{equation}\label{eq:Hprime}
  \frac{\mathrm d}{\mathrm dt}\mf L_{\alpha,\varepsilon}(f(t))
  =-\mathcal D_{\alpha,\varepsilon}(f(t)),
\end{equation}
where the dissipation functional is, for $f \in D(\mathcal{L})\cap L^2_0(\mu)$,
\begin{equation}\label{eq:defD}
  \mathcal D_{\alpha,\varepsilon}(f)
  :=
  -\langle \mc Lf,f\rangle_{L^2(\mu)}
  +\varepsilon\Bigl(
      \langle \mc A_\alpha\mc Lf,f\rangle_{L^2(\mu)}
      +\langle \mc A_\alpha f,\mc Lf\rangle_{L^2(\mu)}
    \Bigr).
\end{equation}

We now estimate $\mathcal D_{\alpha,\varepsilon}(f)$ from below for $f\in C_c^\infty(\RR^{2d})\cap L^2_0(\mu)$.
We write
\begin{equation}\label{eq:decomfunc}
  f_S:=\Pi_v f,\qquad f_F:=(1-\Pi_v)f,
\end{equation}
where $\int_{\RR^d}f_S\ud\mu_x  =\int_{\RR^{2d}}f\ud\mu=0$ by $f\in L^2_0(\mu)$. The subscripts stand for slow and fast components: $f_S$ lies in the kernel of the velocity dissipation, while $f_F$ is controlled by the Gaussian Poincar\'e inequality in the velocity variable.
Since $\mc{L}_a$ is antisymmetric and $\mc{L}_s = -\nabla_v^*\nabla_v$, the uncorrected dissipation is
\begin{equation*}
   -\langle \mc Lf,f\rangle_{L^2(\mu)}
  =\gamma\|\nabla_v f\|_{L^2(\mu)}^2
  \ge \gamma\|f_F\|_{L^2(\mu)}^2.
\end{equation*}
It remains to control the $\varepsilon$-correction term in \eqref{eq:defD}, which we split into
its $\mathcal{L}_s$ and $\mathcal{L}_a$ contributions (dropping the overall factor $\varepsilon$):
\begin{align*}
   \gamma \Bigl(\ip{\mc{A}_\alpha\mc{L}_s f}{f}_{L^2(\mu)} + \ip{\mc{A}_\alpha f}{\mc{L}_s f}_{L^2(\mu)}\Bigr) + \Bigl(\ip{\mc{A}_\alpha \mc{L}_a f}{f}_{L^2(\mu)} + \ip{\mc{A}_\alpha f}{\mc{L}_a f}_{L^2(\mu)}\Bigr).
\end{align*}

For the $\mathcal{L}_s$ contribution, noting  $\mathcal{A}_\alpha = \Pi_v\mathcal{A}_\alpha$ by \eqref{eq:basic-identities1} and $\Pi_v\mathcal{L}_s = 0$, we have
\begin{equation*}
    \langle \mathcal{A}_\alpha f, \mathcal{L}_s f\rangle_{L^2(\mu)} = \langle \Pi_v\mathcal{A}_\alpha f, \mathcal{L}_s f\rangle_{L^2(\mu)} = \langle \mathcal{A}_\alpha f, \Pi_v\mathcal{L}_s f\rangle_{L^2(\mu)} = 0\,.
\end{equation*}
For the other term
$\langle \mathcal{A}_\alpha\mathcal{L}_s f, f\rangle_{L^2(\mu)}$, one can write, by \eqref{eq:basic-identities1} again with \eqref{eq:decomfunc},
\begin{equation} \label{auxeq:second}
    \langle \mathcal{A}_\alpha\mathcal{L}_s f, f\rangle_{L^2(\mu)} = \langle \mathcal{A}_\alpha\mathcal{L}_s f_F, f_S\rangle_{L^2(\mu)}.
\end{equation}
Moreover, note
\begin{equation*}
   \mc A_\alpha^* f = \mc L_a\Pi_v(\alpha-\mc L_{\mathrm o})^{-1}\Pi_v f = v\cdot\nabla_x(\alpha-\mc L_{\mathrm o})^{-1}\Pi_vf,
\end{equation*}
where the inverse is again understood on $L^2_0(\mu_x)$ when $\alpha=0$. We hence have
\begin{equation*}
    \mc L_s\mc A_\alpha^* f = -\mc A_\alpha^* f\,,
\end{equation*}
by $\mc{L}_s(v) = - v$. It then follows from \eqref{auxeq:second} that
\begin{equation*}
    \langle \mc A_\alpha\mc L_s f,f\rangle_{L^2(\mu)}
  =-\langle f_F,\mc A_\alpha^*f_S\rangle_{L^2(\mu)}.
\end{equation*}
The adjoint $\mc A_\alpha^*$ has the same operator norm as $\mc A_\alpha$, and therefore \eqref{eq:A-bound} gives
\begin{equation}\label{eq:fd-bound}
  |\langle \mc A_\alpha\mc L_s f,f\rangle_{L^2(\mu)}|
  \le c_\alpha\|f_F\|_{L^2(\mu)}\|f_S\|_{L^2(\mu)}.
\end{equation}

For the $\mc{L}_a$ contribution, since $\mc{A}_\alpha = \Pi_v\mc{A}_\alpha$ and $\mc{A}_\alpha\Pi_v = 0$
by~\eqref{eq:basic-identities1}--\eqref{eq:basic-identities2}, we have
\begin{align*}
     & \ip{\mc{A}_\alpha\mc{L}_a f}{f}_{L^2(\mu)}+\ip{\mc{A}_\alpha f}{\mc{L}_a f}_{L^2(\mu)} \\ =  &  \ip{\mc{A}_\alpha\mc{L}_a f}{\Pi_vf}_{L^2(\mu)} + \ip{\mc{A}_\alpha (1 - \Pi_v) f}{\Pi_v \mc{L}_a f}_{L^2(\mu)} \\
      = &  \ip{\mc{A}_\alpha\mc{L}_a f_S}{f_S}_{L^2(\mu)}
        + \ip{\mc{A}_\alpha\mc{L}_a f_F}{f_S}_{L^2(\mu)}
        + \ip{\mc{A}_\alpha f_F}{\mc{L}_a f_F}_{L^2(\mu)},
\end{align*}
where in the second equality we have used \eqref{eq:decomfunc} together with $\Pi_v\mc{L}_a f = \Pi_v\mc{L}_a f_F$ (which follows from \eqref{eq:ops2}). The first term is coercive: by~\eqref{eq:basic-identities3},
\begin{equation}\label{eq:slow-coercive}
  \begin{aligned}
  \langle \mc A_\alpha\mc L_a f_S,f_S\rangle_{L^2(\mu)}
 & =
  \langle(\alpha-\mc L_{\mathrm o})^{-1}(-\mc L_{\mathrm o})f_S,f_S\rangle_{L^2(\mu)}
 \\ & \ge
 \inf_{\lambda\ge m}\frac{\lambda}{\alpha+\lambda}
  \|f_S\|_{L^2(\mu)}^2
  = \frac{m}{m+\alpha}\|f_S\|_{L^2(\mu)}^2
  = s_\alpha\|f_S\|_{L^2(\mu)}^2,
  \end{aligned}
\end{equation}
by $f_S\in L^2_0(\mu_x)$ and $-\mc{L}_{\mathrm{o}}|_{L^2_0(\mu_x)}\ge m$.
The remaining two terms are controlled by Lemma~\ref{lem:A-bounds}:
\begin{equation*}
    \bigl|\ip{\mc{A}_\alpha\mc{L}_a f_F}{f_S}_{L^2(\mu)}\bigr|
  \le C_\alpha\,\norm{f_F}_{L^2(\mu)}\norm{f_S}_{L^2(\mu)}\,,\quad \bigl|\ip{\mc{A}_\alpha f_F}{\mc{L}_a f_F}_{L^2(\mu)}\bigr| \le \norm{f_F}_{L^2(\mu)}^2\,.
\end{equation*}
Therefore, we obtain
\begin{equation}\label{est:2ndtermb}
    \ip{\mc{A}_\alpha\mc{L}_a f}{f}_{L^2(\mu)} +\ip{\mc{A}_\alpha f}{\mc{L}_a f}_{L^2(\mu)} \ge s_\alpha\, \norm{f_S}_{L^2(\mu)}^2 -  C_\alpha\,\norm{f_F}_{L^2(\mu)}\norm{f_S}_{L^2(\mu)} - \norm{f_F}_{L^2(\mu)}^2\,.
\end{equation}

Combining \eqref{eq:defD} with the estimates \eqref{eq:fd-bound} and \eqref{est:2ndtermb}, we obtain
\begin{equation}\label{eq:quadratic-form}
 \mathcal D_{\alpha,\varepsilon}(f)  \ge (\gamma-\varepsilon)\|f_F\|_{L^2(\mu)}^2 +\varepsilon s_\alpha\|f_S\|_{L^2(\mu)}^2 -\varepsilon(\gamma c_\alpha+C_\alpha)\|f_F\|_{L^2(\mu)}\,\|f_S\|_{L^2(\mu)}.
\end{equation}
To obtain an explicit lower bound, we reformulate this inequality in matrix form. Letting
\begin{equation*}
    M:=\gamma c_\alpha+C_\alpha,
\end{equation*}
we obtain
\begin{equation}\label{eq:matrix}
 \mathcal D_{\alpha,\varepsilon}(f)
  \ge
  \begin{pmatrix}\|f_F\|_{L^2(\mu)} & \|f_S\|_{L^2(\mu)}\end{pmatrix}
 G_\varepsilon
  \begin{pmatrix}\|f_F\|_{L^2(\mu)}\\\|f_S\|_{L^2(\mu)}\end{pmatrix},
  \qquad
    G_\varepsilon
  :=
  \begin{pmatrix}
    \gamma-\varepsilon & -\varepsilon M/2\\
    -\varepsilon M/2 & \varepsilon s_\alpha
  \end{pmatrix}.
\end{equation}

We next derive an explicit decay rate by bounding the smallest eigenvalue of
$G_\varepsilon$ from below. First, by Sylvester's criterion, $G_\varepsilon$ is positive definite if and only if both leading principal minors are positive: $\gamma - \varepsilon > 0$, and
\begin{equation*}
   \det G_\varepsilon  = \frac{\varepsilon}{4}  \Bigl(4s_\alpha\gamma  -\varepsilon(4s_\alpha+M^2)\Bigr) > 0\,,
\end{equation*}
where the latter condition is equivalent to
\begin{equation} \label{eqcond1}
    0<\varepsilon< \frac{4s_\alpha\gamma}{4s_\alpha+M^2}\,,
\end{equation}
which implies $\gamma - \varepsilon > 0$. That is, $G_\varepsilon$ is positive definite if and only if \eqref{eqcond1} holds. Noting that $\lambda_{\min}(G)\ge \det G/\operatorname{tr}G$ for any positive-definite $2\times2$ matrix $G$, we have
\begin{equation} \label{auxeqeigG}
    \mathcal D_{\alpha,\varepsilon}(f) \ge \lambda_{\min}(G_\varepsilon)  \|f\|_{L^2(\mu)}^2 \ge \frac{\det G_\varepsilon}{\tr G_\varepsilon}  \|f\|_{L^2(\mu)}^2\,.
\end{equation}
We choose the midpoint $\varepsilon_*$ of the interval \eqref{eqcond1}, i.e.,
\begin{equation*}
   \varepsilon_* := \frac{2s_\alpha\gamma}{4s_\alpha+M^2}.
\end{equation*}
Then, we have, by $0<s_\alpha\le1$,
\begin{equation*}
   \lambda_{\min}(G_{\varepsilon_*})
  \ge
  \frac{\det G_{\varepsilon_*}}
       {\operatorname{tr}G_{\varepsilon_*}}
  =
  \frac{s_\alpha^2\gamma}
       {M^2+2s_\alpha+2s_\alpha^2}
  \ge
  \frac{s_\alpha^2\gamma}
       {M^2+4s_\alpha}
  =(1+q_{\alpha,\gamma})\Lambda_{\alpha,\gamma},
\end{equation*}
with
\begin{equation*}
   q_{\alpha,\gamma}
  =2c_\alpha\varepsilon_*
  =
  \frac{4s_\alpha\gamma c_\alpha}
  {4s_\alpha+(\gamma c_\alpha+C_\alpha)^2}\,, \quad \Lambda_{\alpha,\gamma}
=\frac{\varepsilon_*s_\alpha}{2(1+q_{\alpha,\gamma})}\,.
\end{equation*}
It follows from \eqref{auxeqeigG} that, for $f\in C_c^\infty(\RR^{2d})\cap L^2_0(\mu)$,
\begin{equation}\label{eq:core-diss}
  \mathcal D_{\alpha,\varepsilon_*}(f) \ge (1+q_{\alpha,\gamma})\Lambda_{\alpha,\gamma}
  \|f\|_{L^2(\mu)}^2.
\end{equation}
We next extend \eqref{eq:core-diss} to every $f\in D(\mc L)\cap L^2_0(\mu)$.  Since $\mc A_\alpha$ is bounded by Lemma~\ref{lem:A-bounds}, the right-hand side of \eqref{eq:defD} is continuous in the graph norm of $\mc L$: if $f_n\to f$ and $\mc Lf_n\to\mc Lf$ in $L^2(\mu)$, then every term in
\eqref{eq:defD} converges.  Now let
$f\in D(\mc L)\cap L^2_0(\mu)$ and choose $\psi_n\in C_c^\infty(\RR^{2d})$ converging to $f$ in the graph norm of $\mc L$.  Fix $\chi\in C_c^\infty(\RR^{2d})$ with $\int\chi\ud\mu=1$ and set
\begin{equation*}
  \phi_n:=\psi_n-\left(\int\psi_n\ud\mu\right)\chi.
\end{equation*}
Then $\phi_n$ is smooth, compactly supported, and mean-zero. We hence conclude
\begin{equation*}
   \phi_n\to f,
  \qquad
  \mc L\phi_n  =\mc L\psi_n-\left(\int\psi_n\ud\mu\right)\mc L\chi
  \to \mc Lf \quad\text{in }L^2(\mu),
\end{equation*}
by $\int\psi_n\ud\mu\to0$ and $\chi\in D(\mc L)$. Passing to the limit in \eqref{eq:core-diss}, using the graph-norm continuity established above, proves \eqref{eq:core-diss} for every $f\in D(\mc L)\cap L^2_0(\mu)$.

We note that $4s_\alpha+(\gamma c_\alpha+C_\alpha)^2 \ge 4\sqrt{s_\alpha}\,(\gamma c_\alpha+C_\alpha)$, and hence
\begin{equation*}
   q_{\alpha,\gamma}  \le \frac{\sqrt{s_\alpha}\,\gamma c_\alpha}{\gamma c_\alpha+C_\alpha}<1.
\end{equation*}
The norm equivalence \eqref{eq:H-equiv} gives
\begin{equation}\label{eq:Hopt}
  \frac{1-q_{\alpha,\gamma}}{2}\|f\|_{L^2(\mu)}^2  \le \mf L_{\alpha,\varepsilon_*}(f)  \le \frac{1+q_{\alpha,\gamma}}{2}\|f\|_{L^2(\mu)}^2.
\end{equation}
Taking $\varepsilon=\varepsilon_*$, by \eqref{eq:core-diss} and the upper bound in \eqref{eq:Hopt}, we have
\begin{equation*}
   \frac{\mathrm d}{\mathrm dt}\mf L_{\alpha,\varepsilon_*}(f(t)) \le
  -2\Lambda_{\alpha,\gamma}\mf L_{\alpha,\varepsilon_*}(f(t)).
\end{equation*}
By Gronwall's lemma, this further implies
\begin{equation*}
     \|f(t)\|_{L^2(\mu)}^2
  \le
  \frac{1+q_{\alpha,\gamma}}{1-q_{\alpha,\gamma}}\,
  \exp\!\left(-2\Lambda_{\alpha,\gamma}t\right)
  \|f_0\|_{L^2(\mu)}^2,
\end{equation*}
which gives \eqref{eq:general-decay} after taking square roots. The estimate extends to arbitrary $f_0\in L^2_0(\mu)$ by density and the strong continuity of the semigroup. This proves \thmref{thm:main}. \qed

\subsection*{Proof of \corref{cor:optimized-shift}} From \eqref{eq:q-lambda}, a direct computation gives
\begin{equation} \label{auxeq:expressgamma}
    \frac{\gamma}{\Lambda_{\alpha,\gamma}} = \frac4{s_\alpha}
  +\left(\frac{C_\alpha}{s_\alpha}\right)^2 +2\gamma\frac{c_\alpha}{s_\alpha} \left(\frac{C_\alpha}{s_\alpha}+2\right) +\gamma^2\left(\frac{c_\alpha}{s_\alpha}\right)^2.
\end{equation}
We also have, by the definition \eqref{eq:main-constants},
\begin{equation*}
     \frac1{s_\alpha}=1+\frac{\alpha}{m},
  \qquad
  \frac{c_\alpha}{s_\alpha}
  =
  \begin{cases}
    \dfrac1{\sqrt m}, & 0\le\alpha\le m,\\[6pt]
    \dfrac{m+\alpha}{2m\sqrt\alpha}, & \alpha\ge m.
  \end{cases}
\end{equation*}
Noting that for $\alpha\ge m$,
\begin{equation*}
   \frac{\mathrm d}{\mathrm d\alpha}  \frac{m+\alpha}{2m\sqrt\alpha}  =\frac{\alpha-m}{4m\alpha^{3/2}}\ge0\,,
\end{equation*}
we obtain that $1/s_\alpha$ is strictly increasing and $c_\alpha/s_\alpha$ is nondecreasing in $\alpha \ge 0$. Moreover, by
\begin{equation*}
      \left(\frac{C_\alpha}{s_\alpha}\right)^2 =\frac2{s_\alpha^2}
   +2K\left(\frac{c_\alpha}{s_\alpha}\right)^2,
\end{equation*}
$C_\alpha/s_\alpha$ is also nondecreasing in $\alpha$. It follows from \eqref{auxeq:expressgamma} that  $\gamma/\Lambda_{\alpha,\gamma}$ is strictly increasing in $\alpha \ge 0$. Thus, for every fixed $\gamma>0$, there holds
\begin{equation*}
  \Lambda_{\alpha,\gamma}<\Lambda_{0,\gamma}\quad(\alpha>0),\qquad   \sup_{\alpha>0}\Lambda_{\alpha,\gamma}  =\lim_{\alpha\downarrow0}\Lambda_{\alpha,\gamma}  =\Lambda_{0,\gamma}.
\end{equation*}
At $\alpha = 0$, we have $s_0=1$, $c_0=1/\sqrt m$, and $C_0=\sqrt{2+2K/m}$, which gives \eqref{eq:shift-optimized-rate}.  The two
limiting regimes in \eqref{eq:shift-optimized-asymptotics} are immediate. Finally, to optimize over $\gamma>0$, we consider $1/\Lambda_{0,\gamma}$ and find
\begin{equation*}
     \frac1{\Lambda_{0,\gamma}} =
  \frac{6+\frac{2K}{m}}{\gamma}
  +\frac2{\sqrt m}\left(2+\sqrt{2+\frac{2K}{m}}\right)
  +\frac\gamma m,
\end{equation*}
which is minimized at $\gamma_*=\sqrt{6m+2K}$. This completes the proof. \qed


\end{document}